\documentclass[12pt]{article}
\usepackage{graphicx, amsmath, amssymb, amsthm}
\usepackage{color, mathrsfs, wasysym, setspace, mdwlist, calc,float, hyperref, comment, enumitem, tikz}
\usetikzlibrary{shapes.geometric}
\usetikzlibrary{decorations.markings}

\title{(Non-)Recognizing Spaces for Stable Subgroups}
 \author{Sahana Balasubramanya, Marissa Chesser, Alice Kerr, \\ 
Johanna Mangahas, Marie Trin}

\date{\today}

\newcommand{\z}{\mathbb{Z}}
\newcommand{\acts}{\curvearrowright}

\newcommand{\Ga}{\Gamma}

\newtheorem{thm}{Theorem}[section]
\newtheorem{cor}[thm]{Corollary}
\newtheorem{lem}[thm]{Lemma}
\newtheorem{prop}[thm]{Proposition}

\newtheorem{ques}[thm]{Question}

\theoremstyle{definition}
\newtheorem{defn}[thm]{Definition}
\newtheorem{ex}[thm]{Example}
\theoremstyle{remark}
\newtheorem{rem}[thm]{Remark}

\begin{document}

\maketitle
\abstract{In this note, we consider the notion of what we call \emph{recognizing spaces} for stable subgroups of a given group. When a group $G$ is a mapping class group or right-angled Artin group, it is known that a subgroup is stable exactly when the largest acylindrical action $G \acts X$ provides a quasi-isometric embedding of the subgroup into $X$ via the orbit map.  In this sense the largest acylindrical action for mapping class groups and right-angled Artin groups provides a recognizing space for all stable subgroups. In contrast, we construct an acylindrically hyperbolic group (relatively hyperbolic, in fact) whose largest acylindrical action does not recognize all stable subgroups.}

\section{Introduction}
The quasiconvex subgroups of a hyperbolic group are geometrically natural in that they are exactly the finitely generated quasi-isometrically embedded subgroups (meaning, their geometry reflects that of the ambient group), and consequently they are hyperbolic in their own right \cite{Gromov}.  These ideas from Gromov have been widely studied in geometric group theory, and the notion of quasiconvexity now has multiple generalizations to broader settings; see, for example, \cite{FM, DT, Gen, Tran, AM}.

In this note we are interested in the generalization to \emph{stable subgroups} of \emph{acylindrically hyperbolic groups}, specifically those groups admitting a \emph{largest acylindrical action}; exact definitions of these terms appear in Sections \ref{subsec:stablesubgrps}, \ref{subsec:ahgrps}, and \ref{subsec:largest} respectively.  Informally, a subgroup $H$ of a finitely generated group $G$ is stable when $H$ is quasi-isometrically embedded in $G$ and furthermore, quasi-geodesics from $G$ with common endpoints in $H$ uniformly fellow-travel.  It follows from their definition that stable subgroups are hyperbolic, while the ambient $G$ typically is not.  Acylindrical hyperbolicity vastly generalizes hyperbolicity for groups, as the property requires only some suitable action of $G$ on some hyperbolic space $X$. The set of all such actions for a given group form a partially ordered set, which in certain cases admits a largest element.

The notion of stable subgroups was introduced by Durham and Taylor \cite{DT} as a group-theoretic characterisation of the convex cocompact subgroups of mapping class groups \cite{FM}. This class of subgroups has turned out to be a very natural one to consider; in mapping class groups they can be equivalently characterised as the undistorted purely pseudo-Anosov subgroups \cite{BestvinaBrombergKentLeininger}, in right-angled Artin groups they are the purely loxodromic subgroups \cite{RAAGstable}, and for general finitely generated groups they are the boundary convex cocompact groups \cite{CordesDurham}. They are also an example of Morse (or strongly quasiconvex) subgroups \cite{Gen,Tran}, and in fact in mapping class groups they are equivalent to the Morse subgroups of infinite index \cite{HKim}.

A common characterisation of stable subgroups is via considering some well understood action of the group on a hyperbolic space; in several cases the stable subgroups have been shown to be exactly those subgroups quasi-isometrically embedded by the action. We are interested in a slightly broader scenario, which is the existence of a \emph{recognizing space} for a stable subgroup of a finitely generated group $G$, which we define below. This can be compared to a Morse detectable space, as defined in \cite{Russell2022}.

\begin{defn}
    Let $X$ be a hyperbolic space equipped with a $G$--action. If $H$ is a stable subgroup of $G$, we say $X$ is a \emph{recognizing space} for $H$ if the orbit map quasi-isometrically embeds $H$ into $X$.  We call $X$ a \emph{universal recognizing space} for $G$ if it is  a recognizing space for \emph{all} the stable subgroups of $G$.
\end{defn}

Mapping class groups, right angled Artin groups (RAAGs), and, more generally, hierarchically hyperbolic groups (HHGs) each have a well-understood universal recognizing space. Respectively, these are the curve complex \cite[pg. 3]{Hamenstaedt2005WordHE} \cite[Theorem 1.3]{MR2465691}, extension graph \cite[Theorem 1.1]{RAAGstable} and the top level domain of the hierarchy structure \cite[Theorem B]{MR4215647}. It has recently been shown that the genus two handlebody group also has a universal recognizing space \cite{ChesserHandlebody}, as do CAT(0) groups \cite{CAT0}, and more generally groups acting geometrically on a Morse-dichotomous space \cite{Zbinden,PetytZalloum}.

\begin{rem}
    In each of the above examples, it is known that a subgroup is stable if and only if it is quasi-isometrically embedded in the relevant universal recognising space. This also holds when $G$ acts properly on a proper universal recognising space, by a theorem of Aougab--Durham--Taylor \cite[Theorem 1.1]{ADT}.
\end{rem}

In addition to being HHGs, most mapping class groups and RAAGs are also acylindrically hyperbolic groups. In these cases, the recognizing space admits the largest acylindrical action (see Section \ref{subsec:ahgrps} for details), which does not always exist for an acylindrically hyperbolic group. This prompts the question:

\begin{ques}\label{ques1} If $G$ is an acylindrically hyperbolic group with a largest acylindrical action on $X$, is $X$ a universal recognizing space for $G$?
\end{ques}

In other words, does the largest acylindrical action of an acylindrically hyperbolic group, if it exists, contain quasi-isometrically embedded orbits of all of its stable subgroups?  We consider this question for the class of (non-elementary) \emph{relatively hyperbolic groups}, a class of groups introduced by Gromov \cite{Gromov} that generalizes hyperbolic groups. Roughly speaking, the ``non-hyperbolic pieces" of such a group are confined to subgroups known as peripheral subgroups, and coning off the cosets of these peripheral subgroups in a Cayley graph for the group yields a hyperbolic space called the \emph{relative Cayley graph}. In the case that the peripheral subgroups are neither acylindrically hyperbolic nor virtually cyclic, the action on this space is the largest acylindrical action (see Section \ref{subsec:largest} for details).

Aougab--Durham--Taylor prove equivalence between subgroup stability and quasi-isometric embedding by the orbit map into the relative Cayley graph, in the case where the peripheral subgroups all are one-ended and have linear divergence \cite[Theorem 1.5]{ADT}.  Such peripheral groups cannot be acylindrically hyperbolic, as the latter groups contain Morse elements and thus have superlinear divergence \cite{DMS}.  Thus Question \ref{ques1} has a positive answer for a subset of relatively hyperbolic groups, in addition to the HHGs previously mentioned.  Furthermore, it was observed by Abbott and Chesser that, if a stable subgroup of a relatively hyperbolic group has finite intersection with each conjugate of every peripheral subgroup, then it quasi-isometrically embeds in the relative Cayley graph via the orbit map.\footnote{This follows from \cite{MR2684983} Theorem 1.5, Definition 9.6, and Theorem 9.9, along with the fact that stable subgroups are undistorted in the ambient group.}.  We are therefore motivated to ask the following refined question.

\begin{ques}\label{ques2} For a relatively hyperbolic group whose peripheral subgroups are neither acylindrically hyperbolic nor virtually cyclic, are there stable subgroups not quasi-isometrically embedded into the relative Cayley graph via the orbit map? 
\end{ques}

We consider such subgroups \emph{not recognized} by this action.  We provide an affirmative answer to Question \ref{ques2} by constructing a particular relatively hyperbolic group, thereby answering Question \ref{ques1} in the negative. The group we construct is based on a construction of Ol'shanskii, Osin, and Sapir \cite[Theorem 1.12]{OOS}. More precisely, we prove the following.

\begin{thm}\label{thm:main} There exists a non-elementary relatively hyperbolic (and hence acylindrically hyperbolic) group $G$ such that \begin{enumerate}
    \item $G$ has a largest acylindrical action on its relative Cayley graph; and 
    \item $G$ contains stable subgroups that do not quasi-isometrically embed via the orbit map in the relative Cayley graph.
\end{enumerate} 
\end{thm}

To the best of our knowledge, this is the first such example to be recorded in the literature. This also prompts some questions for further exploration. 

\begin{ques}\label{ques3} Given a relatively hyperbolic group, does it admit a (not necessarily acylindrical) action on some universal recognizing space? 
\end{ques}

\begin{ques}\label{ques4} If not, does there exist some recognising space for each stable subgroup of a relatively hyperbolic group?
\end{ques}

We note here that most of the examples of groups with universal recognising spaces given previously were either HHGs or CAT(0) groups, and therefore were finitely presented \cite{Behrstock2019,Bridson1999}. In contrast, it follows from \cite{OOS} that the counterexample we have built is not finitely presented. This raises further questions:

\begin{ques}\label{ques6} If $G$ is a finitely presented, acylindrically hyperbolic group with a largest acylindrical action on $X$, then is $X$ a universal recognising space for $G$?
\end{ques}

\begin{ques} \label{ques5}
    For a finitely presented relatively hyperbolic group whose peripheral subgroups are neither acylindrically hyperbolic nor virtually cyclic, is its relative Cayley graph a universal recognising space?
\end{ques}

The answers to Questions \ref{ques3} and \ref{ques4} may be hard, as the peripheral subgroups may have no natural action on a hyperbolic space, in contrast with the examples stemming from the HHG case. In particular, Remark \ref{rem:noacyl} shows that even if there is a positive answer to Question \ref{ques4}, in general it is not possible for all of the actions to be acylindrical. Regarding Question \ref{ques5}, to answer it (in the negative) by the same method as for Theorem \ref{thm:main}, we would need to find a finitely presented group which is neither acylindrically hyperbolic nor virtually cyclic, but has an infinite stable subgroup.\\

\noindent \textbf{Acknowledgments:} The authors would like to thank the 2023 Women in Groups, Geometry, and Dynamics program and its organizers for hosting them and providing the opportunity to work together. Our thanks also to Jacob Russell for suggesting helpful references, to Anthony Genevois for the comments contained in Remark \ref{rem:genevois} and to Harry Petyt for raising Question \ref{ques5}. The second author thanks Carolyn Abbott for conversations and collaborations that inspired the train of thought that led to this note.  We also thank the referee for a thorough reading of the paper and helpful comments and suggestions, in particular Question \ref{ques6}.

The third author is supported by EPSRC grant EP/V027360/1 ``Coarse geometry of groups and spaces.'' Through the fourth author, this work was supported by a grant from the Simons Foundation (965204, JM) and by National Science Foundation Grant No. DMS-1812021. The fifth author is supported by the Centre Henri Lebesgue ANR-11-LABX-0020-01 and the Région Bretagne's ARED program.


\section{Preliminaries}\label{sec:prelim} 
For the sake of keeping our exposition as complete as possible, this section details some background information and results that will be used in the proof of Theorem \ref{thm:main}. Note that all metric spaces here are assumed to be geodesic.


\subsection{Acylindrically hyperbolic groups.}\label{subsec:ahgrps}

Informally, an acylindrical action can be thought of as a generalization of a proper action, where points are allowed to have large stabilizers, but long cylinders are not. The notion of acylindricity goes back to Sela's paper \cite{Sela}, where it was considered for groups acting on trees. In the context of hyperbolic spaces (more generally, path-metric spaces), the definition is due to Bowditch \cite{Bowditch2008}.  An acylindrically hyperbolic group \cite{Osin} is a group which acts acylindrically on a hyperbolic space, in a way that is in some sense non-trivial. 

\begin{defn}
	An isometric action by a group $G$ on a metric space $(X,d)$ is \emph{acylindrical} if for every $R\geqslant 0$, there exist $N>0$ and $L>0$ such that for every $x,y\in X$ with $d(x,y)\geqslant L$, we have that
	\begin{equation*}
	   |\{g\in G: d(x,gx)\leqslant R\text{ and } d(y,gy)\leqslant R\}|\leqslant N.
    \end{equation*}
	An acylindrical action by a group $G$ on a hyperbolic space $X$ is \emph{non-elementary} if $G$ is not virtually cyclic and orbits are unbounded. A group is \emph{acylindrically hyperbolic} if it admits a non-elementary acylindrical action on a hyperbolic space.
\end{defn}

Over the last few years, the class of acylindrically hyperbolic groups has received considerable attention. It is broad enough to include many examples of interest, e.g., non-elementary hyperbolic and relatively hyperbolic groups (see Section \ref{subsec:relhyp}), all but finitely many mapping class groups of finite-type surfaces without boundary, non-directly-decomposable RAAGs, outer automorphism groups for free groups of rank at least 2, most $3$-manifold groups, and finitely presented groups of deficiency at least $2$. On the other hand, the existence of a non-elementary acylindrical action on a hyperbolic space is a rather strong assumption, which allows one to prove non-trivial results.  For examples of such results, see \cite{DGO,Osin, osinsurvey}, and references therein.


\subsection{Relatively hyperbolic groups.}
\label{subsec:relhyp}

Roughly speaking, a group is relatively hyperbolic group if, when certain subgroups and their conjugates are ``coned-off,'' the resulting Cayley graph is hyperbolic. This is a more general class than that of hyperbolic groups, which are hyperbolic relative to their trivial subgroup. There are many equivalent definitions of relative hyperbolicity (see for example \cite{Bowditch,GM}). For the purposes of this note, we will use the following as our definition.

\begin{prop}\emph{\cite[Proposition 4.28]{DGO}}\label{prop:relhyp} Let $G$ be a group and $\{H_1,\ldots,H_m\}$ a collection of subgroups of $G$. Then $G$ is hyperbolic relative to $\{H_1,\ldots,H_m\}$ if and only if $\{H_1,\ldots,H_m\}$ is hyperbolically embedded in $G$ with respect to some finite $Y \subset G$.
\end{prop}

Being hyperbolically embedded in $G$ with respect to $Y$ means that $Y \cup H_1\cup\cdots\cup H_m$ is a generating set of $G$, that $\Ga(G,Y \sqcup H_1\sqcup\cdots\sqcup H_m)$ is hyperbolic, and additionally that this latter Cayley graph satisfies a technical condition we will not use here. We refer to \cite[Definition 2.1]{DGO} for the full description. Note that it also follows from a result of Sisto \cite[Theorem 2]{sistoHypEmbed} that hyperbolically embedded subgroups are quasi-convex, and thus stable when they are hyperbolic. 

\begin{defn}
    In the case of Proposition \ref{prop:relhyp}, $\{H_1,\ldots,H_m\}$ are referred to as \emph{peripheral subgroups}, and the Cayley graph $\Ga(G,Y \sqcup H_1\sqcup\cdots\sqcup H_m)$ is called the \emph{relative Cayley graph}.
\end{defn}

The following proposition is well-known (see for example \cite{GM} or \cite{OsinRelative}; the relative Dehn function is $0$ is this case) and will be used in the proof of Theorem \ref{thm:main}. The second statement follows from \cite[Corollary 1.5]{OsinRelative}.

\begin{prop}
\label{prop:freeproduct}
    If $H_1,\ldots,H_m$ are groups, then the free product $H_1\ast\cdots\ast H_m$ is hyperbolic relative to $\{H_1,\ldots,H_m\}$. More specifically, $H_1\ast\cdots\ast H_m$ is hyperbolic relative to $\{H_i: H_i\text{ is not hyperbolic}\}$.
\end{prop}

\begin{ex}
    \label{ex:hypemb}
    \cite[Example 2.2]{DGO}
    Let $G = H \ast \mathbb{Z}$ and $\mathbb{Z} = \langle y \rangle$. Set $Y = \{y\}$. Then $H$ is hyperbolically embedded in $G$ with respect to $Y$. In particular, $\Ga(G, Y \sqcup H)$ is quasi-isometric to an infinite-diameter tree (see Fig.\ref{fig:1b}), whose ``vertices" are copies of the complete Cayley graph $\Ga_H = \Ga(H, H)$.
\end{ex}

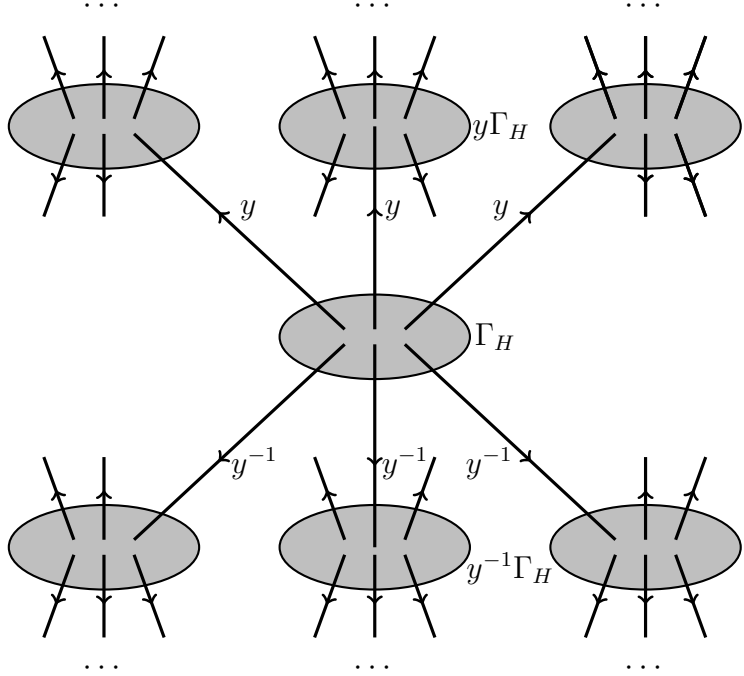
\begin{figure}[ht]
\centering
\begin{tikzpicture}[scale=0.4]

\draw[thick, fill=gray!50] (0, 0) ellipse (90pt and 40pt);
\draw[thick,fill=gray!50] (0, 7) ellipse (90pt and 40pt);
\draw[thick,fill=gray!50] (0, -7) ellipse (90pt and 40pt);
\draw[thick,fill=gray!50] (9, 7) ellipse (90pt and 40pt);
\draw[thick,fill=gray!50] (-9, 7) ellipse (90pt and 40pt);
\draw[thick,fill=gray!50] (-9, -7) ellipse (90pt and 40pt);
\draw[thick,fill=gray!50] (9, -7) ellipse (90pt and 40pt);

\begin{scope}[very thick,decoration={
    markings,
    mark=at position 0.6 with {\arrow{>}}}
    ] 
    
\draw[postaction={decorate}] (0,0.25) -- (0,7);
\draw[postaction={decorate}] (0, -0.25) -- (0,-7);
\draw[postaction={decorate}] (-1,0.25) -- (-8,6.75);
\draw[postaction={decorate}] (1,0.25) -- (8,6.75);
\draw[postaction={decorate}] (-1,-0.25) -- (-8,-6.75);
\draw[postaction={decorate}] (1,-0.25) -- (8,-6.75);

\draw[postaction={decorate}] (-1,7.25) -- (-2,10);
\draw[postaction={decorate}] (0, 7.25) -- (0, 10);
\draw[postaction={decorate}] (1, 7.25) -- (2,10);
\draw[postaction={decorate}] (-1, 6.75) -- (-2, 4);
\draw[postaction={decorate}] (1, 6.75) -- (2,4);

\draw[postaction={decorate}] (8, 7.25) -- (7, 10);
\draw[postaction={decorate}] (9, 7.25) -- (9, 10);
\draw[postaction={decorate}] (10, 7.25) -- (11, 10);
\draw[postaction={decorate}] (9, 6.75) -- (9,4);
\draw[postaction={decorate}] (10, 6.75) -- (11,4);

\draw[postaction={decorate}] (8, 7.25) -- (7, 10);
\draw[postaction={decorate}] (9, 7.25) -- (9, 10);
\draw[postaction={decorate}] (10, 7.25) -- (11, 10);
\draw[postaction={decorate}] (9, 6.75) -- (9,4);
\draw[postaction={decorate}] (10, 6.75) -- (11,4);

\draw[postaction={decorate}] (-8, 7.25) -- (-7, 10);
\draw[postaction={decorate}] (-9, 7.25) -- (-9, 10);
\draw[postaction={decorate}] (-10, 7.25) -- (-11, 10);
\draw[postaction={decorate}] (-10, 6.75) -- (-11,4);
\draw[postaction={decorate}] (-9, 6.75) -- (-9, 4);

\draw[postaction={decorate}] (-8, -7.25) -- (-7, -10);
\draw[postaction={decorate}] (-9, -7.25) -- (-9, -10);
\draw[postaction={decorate}] (-10, -7.25) -- (-11, -10);
\draw[postaction={decorate}] (-9, -6.75) -- (-9, -4);
\draw[postaction={decorate}] (-10, -6.75) -- (-11, -4);

\draw[postaction={decorate}] (-1, -7.25) -- (-2, -10);
\draw[postaction={decorate}] (0, -7.25) -- (0, -10);
\draw[postaction={decorate}] (1, -7.25) -- (2, -10);
\draw[postaction={decorate}] (-1, -6.75) -- (-2, -4);
\draw[postaction={decorate}] (1, -6.75) -- (2, -4);

\draw[postaction={decorate}] (8, -7.25) -- (7, -10);
\draw[postaction={decorate}] (9, -7.25) -- (9, -10);
\draw[postaction={decorate}] (10, -7.25) -- (11, -10);
\draw[postaction={decorate}] (10, -6.75) -- (11, -4);
\draw[postaction={decorate}] (9, -6.75) -- (9, -4);
\end{scope}

\node[black] (c) at (4,0) {$\Gamma_H$};
\node[black] (u) at (4.2,7) {$y \Gamma_H$};
\node[black] (d) at (4.5,-7.7) {$y^{-1}\Gamma_H$};
\node[black] (ar) at (0.6,4.2) {$y$};
\node[black] (ar2) at (1, -4.2) {$y^{-1}$};
\node[black] (ar3) at (-4.2, 4.2) {$y$};
\node[black] (ar4) at (4.2, 4.2) {$y$};
\node[black] (ar5) at (-4, -4.2) {$y^{-1}$};
\node[black] (ar6) at (3.8, -4.2) {$y^{-1}$};
\node[black] (misc) at (0, 11) {$\cdots$};
\node[black] (misc) at (0, -11) {$\cdots$};
\node[black] (misc) at (-9, 11) {$\cdots$};
\node[black] (misc) at (9, 11) {$\cdots$};
\node[black] (misc) at (-9, -11) {$\cdots$};
\node[black] (misc) at (9, -11) {$\cdots$};

\end{tikzpicture}
\caption{$H * \mathbb{Z}$}
\label{fig:1b}
\end{figure}

Relatively hyperbolic groups fit nicely in between the classes of hyperbolic and acylindrically hyperbolic groups, as shown by the following proposition.

\begin{prop}
    \emph{\cite[Proposition 5.2]{Osin}}
    \label{prop:relhypacyl}
    Let $G$ be a group that is hyperbolic relative to a collection of subgroups $\{H_1,\ldots,H_m\}$, and let $Y\subset G$ be a finite set such that $\{H_1,\ldots,H_m\}$ is hyperbolically embedded in $G$ with respect to $Y$. Then $G$ acts acylindrically on the relative Cayley graph $\Ga(G,Y \sqcup H_1\sqcup\cdots\sqcup H_m)$.
\end{prop}

In particular, this means that, so long as $\Ga(G,Y \sqcup H_1\sqcup\cdots\sqcup H_m)$ has infinite diameter and $G$ is not virtually cyclic, $G$ is an acylindrically hyperbolic group. Such relatively hyperbolic groups are referred to as \emph{non-elementary}.


\subsection{Largest acylindrical actions}\label{subsec:largest}

For any group $G$, one can study the set of cobounded acylindrical actions of $G$ on hyperbolic spaces, and ask if this set has a \emph{largest action.} It follows from \cite[Lemma 3.11]{ABO} that this is equivalent to considering the (possibly infinite) generating sets $Y$ of $G$ such that the Cayley graph $\Ga(G,Y)$ is hyperbolic and $G \acts \Ga(G,Y)$ is acylindrical:

\begin{defn}
    Let $G$ be a group, and $Y\subset G$ be a generating set. The action $G \acts \Ga(G, Y)$ is the \emph{largest acylindrical action} if $\Ga(G, Y)$ is hyperbolic and $G \acts \Ga(G, Y)$ is acylindrical, and if for every other generating set $S\subset G$ such that $\Ga(G, S)$ is hyperbolic and $G \acts \Ga(G, S)$ is acylindrical, the identity map of $G$ induces a Lipschitz map $\Ga(G,Y)\to\Ga(G,S)$. 
\end{defn}

Questions about largest actions were considered extensively in \cite{ABO}, and we refer the reader to the  introduction and initial sections of that paper for more precise definitions.

\begin{rem}
    By \cite[Theorem 2.6]{ABO}, if a group $G$ is not acylindrically hyperbolic, then its largest acylindrical action is realised either by taking $G$ itself to be the generating set, or, in the case that $G$ is virtually cyclic, by any finite generating set of $G$. This means that the existence of a largest acylindrical action is non-trivial only if the group is acylindrically hyperbolic.
\end{rem}

Even though there are examples of acylindrically hyperbolic groups that do not admit a largest acylindrical action, (see \cite[Example 7.2 and Theorem 7.3]{ABO}), it turns out that many groups classically studied in geometric group theory do:

\begin{thm}\emph{\cite[Theorem 2.18]{ABO}}\label{thm:ahacc} \emph{\cite[Theorem A]{MR4215647}} The following groups have largest acylindrical actions:
\begin{enumerate}[noitemsep,nolistsep]
\item[(a)] Hyperbolic groups.
\item[(b)] Finitely generated relatively hyperbolic groups whose peripheral subgroups are not acylindrically hyperbolic.
\item[(c)] Hierarchically hyperbolic groups, which includes mapping class groups of orientable surfaces of finite type and right-angled Artin groups.
\item[(d)] Fundamental groups of compact orientable $3$-manifolds with empty or toroidal boundary.
\end{enumerate}
\end{thm}

In all of these cases, the space on which the group has its largest action is well understood.  As mentioned in the introduction, these spaces are universal recognizing spaces for (a) and (c).  For the relatively hyperbolic groups in (b), the largest acylindrical action is on the relative Cayley graph, as long as no peripheral subgroups are virtually cyclic:

\begin{thm}
    \emph{\cite[proofs of Theorem 2.18(a) and Theorem 7.9]{ABO}}
    \label{thm:largestrelhyp}
    Let $G$ be a finitely generated relatively hyperbolic groups whose peripheral subgroups $\{H_1,\ldots,H_m\}$ are neither acylindrically hyperbolic nor virtually cyclic. Let $Y\subset G$ be a finite set such that $\{H_1,\ldots,H_m\}$ is hyperbolically embedded in $G$ with respect to $Y$. The largest acylindrical action of $G$ is on the relative Cayley graph $\Ga(G,Y \sqcup H_1\sqcup\cdots\sqcup H_m)$.
\end{thm}


\subsection{Stable subgroups of relatively hyperbolic groups.}\label{subsec:stablesubgrps}

Stable subgroups of an arbitrary finitely generated group are defined as follows.

\begin{defn}\label{def:stable} Let $G$ be a finitely generated group. A finitely generated subgroup $K \leq G$ is said to be \emph{stable} if these conditions both hold:\begin{enumerate}
    \item $K$ is undistorted in $G$, meaning $K$ quasi-isometrically embeds into $G$ with respect to their respective finite generating sets. 
    \item For any constants $(\lambda,C)$, there is a constant $M >0$ (depending on $\lambda, C$) such that any two $(\lambda,C)$--quasi-geodesics with the same end points in $K$ are $M$--close to each other.
\end{enumerate} 
\end{defn}

When the ambient group is relatively hyperbolic, stable subgroups can also be characterised in terms of how they intersect conjugates of peripheral subgroups.

\begin{thm}\emph{\cite[Corollary 1.10]{Tran}} \label{thm:transferstable} Let $G$ be a finitely generated relatively hyperbolic groups with peripheral subgroups $\{H_1,\ldots,H_m\}$. Let $K$ be a finitely generated undistorted subgroup of $G$. Then the following are
equivalent: \begin{enumerate}
    \item $K$ is stable in $G$
    \item For every peripheral subgroup $H_i$, $K \cap H_i^g$ is stable in $H_i^g$ for each conjugate $H_i^g$
    \item For every peripheral subgroup $H_i$, $K \cap H_i^g$ is stable in $G$ for each conjugate $H_i^g$
\end{enumerate}
\end{thm}

We would like to use this to construct an example of a relatively hyperbolic group with a stable subgroup not recognised by the relative Cayley graph. In this context, one useful property of peripheral subgroups is the following.

\begin{lem}
    \emph{\cite{Bowditch}}
    Let $G$ be hyperbolic relative to $\{H_1,\ldots,H_m\}$. Then for every $g,h\in G$ and $i,j\in\{1,\ldots,m\}$, if the conjugates $H_i^g$ and $H_j^h$ are not equal, then their intersection is finite.
\end{lem}

As any finite subgroup is stable, which can be seen directly from Definition \ref{def:stable}, we observe the following corollary to Theorem \ref{thm:transferstable}.

\begin{cor}
\label{cor:transferstable}
    Let $G$ be a finitely generated relatively hyperbolic groups with peripheral subgroups $\{H_1,\ldots,H_m\}$. Let $K$ be a finitely generated undistorted subgroup of $G$. Suppose that $K$ is a subgroup of some $H_i$, and that $K$ is stable in $H_i$. Then $K$ is stable in $G$.
\end{cor}


\section{Proof of Main Theorem}\label{sec:mainpf}

\begin{proof}[Proof of Theorem \ref{thm:main}] By \cite[Theorem 1.12]{OOS}, there exists a finitely generated group $H$ such that every proper subgroup is infinite cyclic, while $H$ itself is not virtually cyclic. Moreover, for any finite generating set $S$ of $H$, \cite[Theorem 1.12]{OOS} states that every periodic path in the Cayley graph $\Ga(H,S)$ is a Morse quasi-geodesic. In particular, for every nontrivial $q\in H$, we have that $\langle q\rangle$ is Morse and hyperbolic, which implies that $\langle q\rangle$ is stable in $H$ by \cite[Proposition 4.3]{Tran}. 

Consider $G= H \ast \z$, where $y$ is the generator of $\z$. Then $G$ is finitely generated by $\{y \}\cup S$. Recall from Proposition \ref{prop:freeproduct} that $G$ is relatively hyperbolic with peripheral subgroup $H$. By Example \ref{ex:hypemb} and Proposition \ref{prop:relhypacyl}, $G$ is also acylindrically hyperbolic.  However, as the peripheral subgroup $H$ contains no non-abelian free subgroups, $H$ is not acylindrically hyperbolic (as such groups have independent loxodromics acting on a hyperbolic space, which allows one to construct a rank-two free group by the standard ping-pong argument; see for example \cite{Gromov,Osin}). It follows from Example \ref{ex:hypemb} and Theorem \ref{thm:largestrelhyp} that $G$ admits a largest acylindrical action on the relative Cayley graph $\Ga(G, \{y \}\sqcup H)$.

Recall that for every nontrivial $q\in H$, we have that $\langle q\rangle$ is stable in $H$. As $G$ is a free product, any path with endpoints in $\Ga(H,S)$, which also leaves $\Ga(H,S)$, must visit the same point twice. We therefore have that $\Ga(H,S)$ isometrically embeds into $\Ga(G,\{y \}\cup S)$. As $\langle q\rangle$ is a quasi-geodesic in $H$, we automatically get that $\langle q\rangle$ is a quasi-geodesic in $G$, and so is undistorted in $G$. It therefore follows from Corollary \ref{cor:transferstable} that $\langle q\rangle$ is stable in $G$. However, $\langle q \rangle$ is clearly not quasi-isometrically embedded in the relative Cayley graph $\Ga(G, \{y\}\sqcup H)$, as the orbit of $H$ is bounded therein, and $q$ has infinite order.
\end{proof}

\begin{rem}
    \label{rem:noacyl}
    In fact, no acylindrical action of $G$ on a hyperbolic space can recognize the stable subgroup $\langle q \rangle$. Indeed, if $G \acts X$ is an acylindrical action on a hyperbolic space, then it follows that the restriction of the action to $H$ is also acylindrical. As $H$ is not acylindrically hyperbolic, and is not virtually cyclic, the action of $H$ must have bounded orbits. Consequently $\langle q\rangle$ cannot be quasi-isometrically embedded in $X$, as that would imply that $H$ has an unbounded orbit, which is a contradiction. In other words, any positive answer to Question \ref{ques3} in this case would be via a non-acylindrical action.  Regarding Question \ref{ques4}, even if recognizing spaces exist for each individual stable subgroup of $G$, this means that some of the actions will not be acylindrical. 
\end{rem}

\begin{rem}
    Note that the stability of $\langle q\rangle$ in the above proof of Theorem \ref{thm:main} can also be seen directly. The fact that $\langle q\rangle$ is quasi-isometrically embedded in the Cayley graph $\Ga(G,\{y \}\cup S)$ is already addressed directly in the proof; one can also see directly that the uniform fellow-traveling of quasi-geodesics is satisfied. In particular, we may view $\Ga(G,\{y \}\cup S)$ in a manner similar to that in Figure \ref{fig:1b}, with the complete subgraph $\Ga(H,H)$ replaced with copies of $\Ga(H, S)$. This is still a space with a ``tree-like" structure. Suppose $\sigma$ is a quasi-geodesic in $\Ga(G,\{y \}\cup S)$ that starts and ends in $\langle q\rangle\subset \Ga(H,S)$, and let $\langle q \rangle_\sigma$ be the segment of $\langle q\rangle$ with the same endpoints as $\sigma$. The lengths of segments of $\sigma$ that leave $\Ga(H,S)$ are bounded uniformly (depending on the quasi-geodesic constants), as they must eventually return to the same point from where they left $\Ga(H,S)$. As the segments of $\sigma$ that remain in $\Ga(H,S)$ form a quasi-geodesic, and $\langle q\rangle_{\sigma}$ is itself a quasi-geodesic, by the stability of $\langle q\rangle$ in $H$ it follows that $\sigma$ stays uniformly close to $\langle q\rangle_{\sigma}$. The same is true of any other quasi-geodesic with the same constants and endpoints as $\sigma$, and therefore $\langle q\rangle$ is stable in $G$. 
\end{rem}

\begin{rem}\label{rem:genevois}
    We note that the only properties used about the group $H$ from \cite{OOS} is that it is finitely generated and contains Morse elements, but is neither acylindrically hyperbolic nor virtually cyclic. As very few such examples are known, it would be interesting to consider if there exist groups containing Morse elements with even worse ``un-hyperbolic-like" properties. For instance, does there exist a group $A$ (resp. finitely presented group $A$) with Morse elements with the Property $\mathrm{(NL)}$ (i.e. admitting no actions on hyperbolic spaces with loxodromic elements; see \cite{PropNL})? If so, Question \ref{ques4} (resp. Question \ref{ques5}) automatically will have a negative answer by considering the free product $G = A \ast \z$. As it is unclear if the group from \cite{OOS} has property $\mathrm{(NL)}$, we leave such explorations for future work. 
\end{rem}

\bibliographystyle{amsalpha}
\bibliography{bibclean}

\end{document}